\documentclass[12pt]{amsart}
\usepackage[all,cmtip]{xy}
\usepackage{pb-diagram,pb-xy}
\usepackage{amsmath,amscd,amsthm,amsfonts, amssymb,amsxtra}
\usepackage{hyperref}
  \hypersetup{colorlinks=true,citecolor=blue}
 \usepackage[font=scriptsize]{caption}

\CDat

\SelectTips{cm}{}
\newtheorem{thm}{Theorem}[section]
\newtheorem*{un-no-thm}{Theorem}
\newtheorem{cor}[thm]{Corollary}     
\newtheorem{lem}[thm]{Lemma}         
\newtheorem{prop}[thm]{Proposition}

\newtheorem{hypo_no}[thm]{Hypothesis}

\newtheorem{bigthm}{Theorem}

\newtheorem{bigcor}[bigthm]{Corollary}

\newtheorem{bigadd}[bigthm]{Addendum}

\theoremstyle{definition}
\newtheorem{defn}[thm]{Definition}   

\theoremstyle{definition}

\theoremstyle{definition}

\theoremstyle{remark}
\newtheorem{rem}[thm]{Remark}

\newtheorem*{acks}{Acknowledgements}
\newtheorem*{out}{Outline}

\newtheorem{ex}[thm]{Example}

\newtheorem*{convent}{Conventions}

\DeclareMathOperator{\End}{end}

\DeclareMathOperator{\im}{im}
\DeclareMathOperator{\tr}{tr}

\begin{document}
\title[Kirchhoff's theorems in higher dimensions]{Kirchhoff's theorems in higher dimensions
and Reidemeister torsion}

\date{\today}

\author[M.J.~Catanzaro]{Michael J.\ Catanzaro}
\address{Department of Mathematics, Wayne State University, Detroit, MI 48202}
\email{mike@math.wayne.edu}
\author[V.Y.~Chernyak] {Vladimir Y.\ Chernyak}
\address{Department of Chemistry, Wayne State University, Detroit, MI 48202}
\email{chernyak@chem.wayne.edu}
\author[J.R.~Klein]{John R.\ Klein}
\address{Department of Mathematics, Wayne State University, Detroit, MI 48202}
\email{klein@math.wayne.edu}

\begin{abstract} Using ideas from algebraic topology and
statistical mechanics, we generalize Kirchhoff's network and matrix-tree theorems to
finite CW complexes of arbitrary dimension. As an application, we give a formula expressing Reidemeister
torsion as an enumeration of higher dimensional spanning trees.
\end{abstract}

\maketitle
\setlength{\parindent}{15pt}
\setlength{\parskip}{1pt plus 0pt minus 1pt}
\def\bdot{\bold .}
\def\Sp{\bold S\bold p}

\def\vo{\varOmega}
\def\smsh{\wedge}
\def\^{\wedge}
\def\flush{\flushpar}
\def\id{\text{\rm id}}
\def\dbslash{/\!\! /}
\def\codim{\text{\rm codim\,}}
\def\:{\colon}
\def\holim{\text{holim\,}}
\def\hocolim{\text{hocolim\,}}
\def\cal{\mathcal}
\def\Bbb{\mathbb}
\def\bold{\mathbf}
\def\simtwohead{\,\, \hbox{\raise1pt\hbox{$^\sim$} \kern-13pt $\twoheadrightarrow \, $}}
\def\codim{\text{\rm codim\,}}
\def\stableto{\mapstochar \!\!\to}
\let\Sec=\S
\def\Z{\mathbb Z}
\def\Q{\mathbb Q}
\def\R{\mathbb R}
\newcommand{\lra}{\longrightarrow}

\setcounter{tocdepth}{1}
\tableofcontents

\section{Introduction \label{sec:intro}}

Gustav Kirchhoff's results on electrical networks,
which predate Maxwell's theory of electromagnetism,
are a product of the mid-19$^{\text{th}}$ century \cite{Kirchhoff1}, \cite{Kirchhoff2}.
{\it Kirchhoff's network theorem} states that in any resistive network
 there is a unique current satisfying Ohm's law and
Kirchhoff's current and voltage laws, and furthermore this current can be
explicitly computed.  The first complete treatment of the network theorem
is attributed to Hermann Weyl \cite{Weyl} in 1923. By the mid-20$^\text{th}$ century,
algebraic topology provided key ideas leading
to a simple and elegant proof \cite{Eckmann},\cite{Roth},\cite{NS}.
A companion result is {\it Kirchhoff's matrix-tree theorem} which gives a formula
for the number of spanning trees in a finite connected graph (see 
\cite{Moon} for a history of this result).
This paper is an outgrowth of our investigations on the interplay between
algebraic topology and statistical mechanics \cite{CKS1}, \cite{CKS2}, \cite{CKS3}.
Our aim is to generalize Kirchhoff's results to higher dimensions, as
well as to connect these results to the theory of Reidemeister torsion.

\subsection*{A high dimensional network theorem}
Suppose $X$ is a finite connected CW complex of dimension $d$.
Let $C_j(X;\Bbb R)$ denote the cellular chain complex of $X$ with real coefficients
and the standard inner product $\langle \phantom{a},\phantom{a}\rangle$ for which 
the set of $j$-cells, denoted $X_j$, is an orthonormal basis.
In what follows we fix a function $r\: X_d \to \Bbb R_+$; the value of
$r$ at a $d$-cell $b$ is considered to be the resistance of $b$.
Define a linear transformation $R\: C_{d}(X;\R) \to C_d(X;\R)$  by mapping a $d$-cell $b$ to $r_b b$ and extending linearly.
Let $B_{d-1}(X;\R) \subset C_{d-1}(X;\R)$ be the vector subspace of $(d-1)$-boundaries
and let $Z_d(X;\R)\subset C_d(X;\R)$ be the vector subspace of $d$-cycles.

\begin{defn}
A {\it network problem} for $X$ consists of a choice of $p \in B_{d-1}(X;\R)$ and
$q \in Z_d(X;\R)$, respectively called {\it  $(d{-}1)$-boundary current} and
{\it $d$-cycle voltage}.\footnote{When $d=1$, in the terminology of Roth \cite{Roth},
$p$ is a {\it node current} and $q$ is a {\it mesh voltage}, each arising from an external source. Bollob\'as \cite[p.~41]{Bollobas} only considers the case when $q =0$ and $p$ is of the form $p_ii + p_jj$ for a pair of
distinct vertices $i$ and $j$.}
A {\it solution}
consists of  $\mathbf V,\mathbf J \in C_d(X;\R)$ such that

\begin{align}
&\mathbf V = R\mathbf J\, , & \text{(Ohm's law)} \\
&  \partial \mathbf J = p \, , &\text{(current  law)}  \\
&  \langle \mathbf V ,z \rangle = \langle q ,z \rangle \, , \quad \text{ for all }  z \in Z_d(X)\, .  &\text{(voltage  law)}
\end{align}

\end{defn}


To see why
a solution exists,
define a modified inner product $\langle\phantom{-},\phantom{-}\rangle_R$ on $C_d(X;\Bbb R)$
by
\[
\langle b,b'\rangle_R = \langle Rb,b'\rangle
\]
for $b,b' \in X_d$.
Let
\[\partial^\ast_R\: C_{d-1}(X;\Bbb R) \to C_d(X;\Bbb R)
\] denote the formal
adjoint to $\partial$ using the standard inner product on $C_{d-1}(X;\Bbb R)$
and the modified inner product on $C_d(X;\Bbb R)$.
Let $B^d_R(X;\Bbb R)$
be the image of $\partial_R^\ast$ and note
that $B^d_R(X;\Bbb R)$ is the orthogonal complement to $Z_d(X;\Bbb R)$ in $C_d(X;\R)$
with respect to the modified
inner product. Elementary linear algebra implies
$\partial\: B^d_R(X;\Bbb R) \to B_{d-1}(X;\R)$ is an isomorphism.
Consequently, there is a unique $\mathbf  J_0 \in B^1_R(X;\R)$ such that
$\partial \mathbf J_0 = p$. Set $\mathbf V_0 = R\mathbf J_0$. Then
$\langle \mathbf V_0 ,z \rangle = \langle \mathbf J_0,z \rangle_R =  0$ for all $z \in Z_d(X;\R)$.
Let $\mathbf J_1$ be the orthogonal projection of $R^{-1}q$ onto $Z_d(X;\R)$ in
the modified inner product, and set $x = \mathbf J_1- R^{-1}q $.
Then $\langle R\mathbf J_1 - q,z \rangle =
\langle x,z \rangle_R = 0$ for all $z \in Z_d(X;\R)$. Set $\mathbf V_1 = R\mathbf J_1$.  Then $\mathbf J := \mathbf J_0 +
\mathbf J_1$ and $\mathbf V := \mathbf V_0 + \mathbf V_1$ solve the network problem.
It is straightforward to show that this solution is unique.
\medskip

The above solution to the network problem uses
the orthogonal projection of $C_d(X;\Bbb R)$ onto
$Z_d(X;\Bbb R)$
in the modified inner product. In the classical case $d=1$, Kirchhoff gave a formula
expressing the orthogonal projection
as a weighted sum indexed over  the set of spanning trees
of $X$.  To get an explicit formula
in higher dimensions we will need a notion of spanning tree.

\begin{defn}
\label{def:spa-t} Assume as above that $X$ is a connected finite CW complex of dimension $d$.
A {\it spanning tree} for $X$ is a subcomplex $T$ such that
\begin{itemize}
\item $H_d(T;\Bbb Z) = 0$,
\item $\beta_{d-1}(T) = \beta_{d-1}(X)$, where $\beta_k(X)$ denotes the $k$-th Betti number,
\item $X^{(d-1)} \subset T$, where $X^{(k)}$ is the $k$-skeletion of $X$.
\end{itemize}
\end{defn}

\begin{rem} We will show in the next section such spanning trees exist.
The reader should be aware that the literature contains sundry notions of ``high dimensional spanning tree.''\footnote{For example, \cite{DKM}
replaces condition (2) with  the requirement that the reduced Betti number $\tilde \beta_{d-1}(T)$ is trivial. This implies
$\tilde\beta_{d-1}(X)$ is trivial as well, so the notion of spanning tree
in \cite{DKM} does not apply
to a general finite complex $X$. See \cite{Petersson} for a detailed discussion of the various notions.}
Note that when $d=1$, our definition reduces to
the classical notion of spanning tree.
\end{rem}

\begin{defn}
For a spanning tree $T$ of $X$, define a linear transformation
\[
\bar T\:
C_{d}(X;\Bbb R) \to Z_d(X;\Bbb R)
\]
 as follows: Let $b$  be a $d$-cell. If $b$ is contained  in $T$ then
we set $\bar T(b) = 0$. Otherwise, note that $H_d(T\cup b;\Z) = Z_d(T\cup b;\Z)$ is infinite cyclic. Let $c$ be a generator. Set $t_b = \langle c, b\rangle$ (this is always
non-zero). Then $\bar T(b) := c/t_b$, is a real $d$-cycle of $X$.
It is easy to see that $\bar T(b)$ is independent of the choice of $c$.
\end{defn}

Let $\theta_T$ denote the order of the torsion subgroup of
$H_{d-1}(T;\Bbb Z)$ and define the {\it weight of $T$} to be the positive real number
\[
w_T := \theta_T^2\prod_{b\in T_{d}} r_b^{-1}\, .
\]

\begin{bigthm}[Higher Projection Formula] \label{bigthm:network} With respect to the modified inner product
$\langle \phantom{a} ,\phantom{a} \rangle_R$,
the orthogonal projection $C_{d}(X;\Bbb R) \to Z_d(X;\Bbb R)$ is given by
$$
\tfrac{1}{\Delta}\sum_{T} w_T\bar T\, ,
$$
where the sum is over all spanning trees, and  $\Delta = \sum_T w_T$.
\end{bigthm}

Let $\partial^*\: C_{d-1}(X;\R) \to C_d(X;\R)$ be the formal adjoint to 
the boundary operator with respect to the standard inner product. Define
$B^{d}(X,\R)$ to be the image of $\partial^*$. Then we have

\begin{bigadd}[Higher Network Theorem] \label{add:Kirchhoff_formula}
Given a vector $\mathbf V \in C_d(X;\R)$, there is only one vector
$z\in Z_d(X;\R)$ such that $\mathbf V - Rz \in B^{d}(X,\R)$. Furthermore, for each $d$-cell
$b$, we have
\[
\langle z,b\rangle = \tfrac{1}{\Delta} \sum_T \tfrac{w_T}{r_b} \langle \mathbf V,\bar T(b)\rangle \, .
\]
\end{bigadd}

\begin{rem} In classical network terminology ($d=1)$, $\langle \mathbf V,b \rangle$
is the {\it voltage source} on branch $b$ and
 $\langle z,b\rangle$ is the {\it current} resulting in branch $b$ (see \cite{Roth},\cite{NS}).
\end{rem}

\subsection*{A high dimensional matrix-tree theorem}
The classical matrix-tree theorem enumerates
the number of spanning trees of a graph. In higher dimensions, the best we can achieve is
an expression for $\sum_{T} \theta^2_T$.

Observe that $B_{d-1}(X;\Bbb R)$ is an invariant subspace of the operator
\[
\partial \partial_R^*\: C_{d-1}(X;\Bbb R) \to C_{d-1}(X;\Bbb R).
\]
Let
\[
\cal L^R \: B_{d-1}(X;\Bbb R) @> \cong >> B_{d-1}(X;\Bbb R)
\]
denote the associated restriction.

\begin{bigthm}[Higher Weighted Matrix-Tree Theorem] \label{thm:higher-matrix-tree} We have
\[
\det \cal L^R = \gamma_X\sum_T w_T \, ,
\]
where the sum is indexed over all spanning trees of $X$, and the normalizing
factor is given by
\[
\gamma_X = \frac{\mu_X}{\theta_X^2}\, ,
\]
where $\mu_X \in \Bbb N$ is the square of the covolume of the lattice $B_{d-1}(X;\Bbb Z) \subset
B_{d-1}(X;\Bbb R)$ with respect to the restriction of the standard
inner product of $C_{d-1}(X;\Bbb R)$  and $\theta_X$ is the order of the torsion subgroup
of $H_{d-1}(X;\Bbb Z)$.
\end{bigthm}

The unweighted case when $r\: X_{d} \to \Bbb R_+$ is constant with value 1 is worth singling
out, as it gives rise
to the operator \[
\cal L = \partial\partial^\ast \:B_{d-1}(X;\Bbb R)@> \cong >> B_{d-1}(X;\Bbb R)\, .
\]
In this case $w_T = \theta^2_X$ and Theorem  \ref{thm:higher-matrix-tree} becomes

\begin{bigcor}[Higher Matrix-Tree Theorem] \label{cor:higher-matrix-tree} For $\cal L$ as above, we have
 \[ \det \cal L = \gamma_X\sum_T \theta^2_T\, .\]
\end{bigcor}

\begin{rem} Variations of Corollary \ref{cor:higher-matrix-tree}
have appeared in \cite{Kalai}, \cite{Petersson}, \cite{DKM} and \cite{Lyons} (note: all but
the last reference assume additional conditions on $X$, and each work
utilizes its own notion of spanning tree).
When $d=1$, we have $\theta_T = 1 =\theta_X$ and
$\mu_X$ is the number of vertices of $X$, so we obtain the classical Kirchhoff matrix-tree theorem. Theorem \ref{thm:higher-matrix-tree} is actually a special case of a more general result, Theorem \ref{thm:general-higher-matrix-tree} below.
\end{rem}

Corollary \ref{cor:higher-matrix-tree} admits the following 
simpler reformulation (cf. Corollary \ref{cor:Lmu} below).

\begin{bigadd} \label{addC}
\[
\det \cal L = \sum_T \cal L^T = \sum_T \mu_T \, ,
\]
where $\cal L^T = \partial\partial^*_T\: B_{d-1}(T;\Bbb R)@> \cong >> B_{d-1}(T;\Bbb R)$.
\end{bigadd}

\subsection*{Reidemeister torsion counts spanning trees}
Milnor \cite{Milnor} introduced  the notion of Reidemeister torsion $\tau(C_\ast)$ of
a not necessarily acyclic finite chain complex $C_\ast$ over a field
in which a preferred basis is chosen for $C_\ast$ as well as its homology. When $C_\ast$ is
the real chain complex of a finite CW complex $X$, we will establish a connection between
the torsion and the enumeration of spanning trees on the skeleta of $X$.

Suppose $X$ is a  finite, connected CW complex. We give $C_\ast(X;\R)$ the preferred
basis given by its set of cells. We also choose a basis for $H_\ast(X;\R)$
by selecting a basis for the torsion free part of each integral homology group
$H_\ast(X;\Z)$. Such a basis is called a {\it combinatorial basis} for the homology 
and we will denote it by $\mathfrak h$. The definition of 
Reidemeister torsion $\tau(X;\mathfrak h)$ is given in \S\ref{sec:torsion-tree} below. 

For $k \ge 0$, we define the following quantities:
\begin{itemize}
\item $\cal T_k =$ the set of spanning trees of $X^{(k)}$ (for this we require $k >0$).
\item $\mu_{k} =$ the square of the covolume of the lattice 
$B_{k}(X;\Z) \subset B_{k}(X;\R)$ with
respect to the inner product given by restricting the standard inner product on
 $C_{k}(X;\R)$. 
 \item $H_k(X;\Z)_0 =$ the image of the evident homomorphism
 $H_k(X;\Z) \to H_k(X;\R)$.
 \item $\eta_k =$ the square of the covolume of the lattice
 $H_k(X;\Z)_0 \subset H_k(X;\R)$, where we give $H_k(X;\R)$ the inner product
 defined by identifying the latter with the orthogonal complement of $B_k(X;\R) \subset Z_k(X;\R)$ using the inner product arising from the standard one on $C_k(X;\R)$.
\item $\theta_{k}=$ 
the order of the torsion subgroup of $H_{k}(X;\Z)$. 
\end{itemize}
With respect to the above, we set
\[
\delta_k := \frac{\eta_k\mu_k}{\theta^2_k}.
\]
Then $\delta_k$ is defined entirely in terms of $X$.

\begin{bigthm}[Torsion-Tree Theorem]\label{thm:torsion-tree}
For a finite, connected CW complex $X$, we have
\begin{equation*}
\tau^2(X;\mathfrak h) = 
{\prod_{k \ge 0} 
 (\delta_{k} \sum_{T\in \cal T_{k+1}} \theta^2_T})^{(-1)^k}\,  ,
\end{equation*}
where $\theta^2_T$ denotes the order of
of the torsion subgroup of $H_{k}(T;\Z)$ for $T\in \cal T_{k+1}$.
\end{bigthm}

\begin{convent}
We assume the reader is familiar with basic
linear algebra as well as a first year course on algebraic topology.
The topological spaces of this paper are equipped with 
preferred CW structure and when we write $H_\ast(X;A)$, we mean cellular homology 
with coefficients in an abelian group $A$ (in practice, $A$ is either $\Z$ or $\R$).
If $X$ is a CW complex, we write $X^{(k)}$ for its $k$-skeleton and $X_k$ for its
set of $k$ cells. Thus 
\[
X^{(k)} = X^{(k-1)} \cup (X_k \times D^k)\, ,
\]
where the union is amalgamated along the attaching map $X \times S^{k-1} \to X^{(k-1)}$. The 
$k$-th Betti number $\beta_k(X)$ is defined to be the rank of the vector space $H_k(X;\R)$.
If $A$ is a commutative ring, then the $k$-th real chain group $C_k(X;A)$ is by definition the relative homology group $H_k(X^{(k)},X^{(k-1)};A)$, which 
is just the free $A$-module having basis $X_k$. 
\end{convent}

\begin{out} In \S\ref{sec:spanning-trees} we develop basic results about higher
dimensional spanning trees. In \S\ref{sec:proof-thm-network}
we prove Theorem \ref{bigthm:network} and Addendum \ref{add:Kirchhoff_formula}.
In \S\ref{sec:weak-matrix-tree} we prove Theorem  \ref{thm:higher-matrix-tree} up to identification of the normalizing constant $\gamma_X$. In \S\ref{sec:low-temp} we introduce the low temperature limit and use it to show that for sufficiently well-behaved $W$, 
the determinant
of $\cal L$ tends in the low temperature limit to the determinant of $\cal L^T$,
where the latter is defined using a spanning tree $T$ in place of $X$. This result is
employed in \S\ref{app:extras} to identify $\gamma_X$, thereby completing the proof
of Theorem  \ref{thm:higher-matrix-tree}; in so doing we
 generalize Theorem \ref{thm:higher-matrix-tree} to Theorem
\ref{thm:general-higher-matrix-tree}.
Lastly, in \S\ref{sec:torsion-tree}, we outline Milnor's definition of Reidemeister
torsion and prove Theorem \ref{thm:torsion-tree}. Also, in  Theorem
\ref{thm:last}, we obtain a different expression for
the Reidemeister torsion that is expressed in terms of both spanning tree and homology truncation data for $X$.
\end{out}

\begin{acks} The authors
thank the Los Alamos Center for Nonlinear Studies and the T-4 division for partially supporting this research. This material is based upon work supported by the National
Science Foundation under Grant Nos.\ CHE-1111350, DMS-0803363
and DMS-1104355. The third author wishes to thank Paul Kirk for discussions related to
the beginning of \S\ref{sec:torsion-tree} as well as Paul Penfield for providing him with a translation to Weyl's article \cite{Weyl}.
\end{acks}

\section{Spanning Trees in higher dimensions \label{sec:spanning-trees}}

\begin{defn} Let $X$ be a finite connected CW complex.
\label{def:ess} A $k$-cell $b \in X_k$ is said to be {\em essential} if there exists a
$k$-cycle $z \in Z_k(X;\R)$
such that $\langle z,b\rangle \neq 0$.
\end{defn}

\begin{lem} \label{lem:add_or_remove}
\label{lem:add-rem} Assume in addition $X$ has dimension $d$. Then adding or removing an essential $d$-cell from $X$ increases or decreases $\beta_d(X)$ by one, respectively, and fixes $\beta_{d-1}(X)$.
\end{lem}

\begin{proof}
Construct a decreasing filtration $Y^i$ on $X$ by removing the $d$-cells of $X$ one at a time,
$X = Y^n \supset Y^{n-1} \supset  \ldots \supset Y^0 = X_{d-1}$.  Then for $1\le j \le n$,
we have
an exact sequence in homology
\[
  0 \to H_d (Y^{j-1}) \to H_d(Y^j) \stackrel{\partial_*}{\lra} \Z \to H_{d-1}(Y^{j-1}) \to H_{d-1}(Y^j) \to 0
\]
The above factors into two short exact sequences
\begin{align*}
& 0 \to H_d(Y^{j-1}) \to H_d(Y^j) \to \im \partial_* \to 0 \\
& 0 \to \Z/\im \partial_*  \to H_{d-1}(Y^{j-1}) \to H_{d-1}(Y^j) \to 0\, ,
\end{align*}
where $\im \partial_*$ is the image of $\partial_\ast$.
 If the attached cell is essential, then
$\im \partial_*$ is a nontrivial subgroup of $\Z$.
Therefore, the first sequence implies $\beta_d(Y^j) = \beta_d(Y^{j-1}) + 1$,
while the second implies $\beta_{d-1}(Y^j) = \beta_{d-1}(Y^{j-1})$. We may view $Y^j$ as a complex
with an additional essential cell, or $Y^{j-1}$ as a complex with an essential cell removed.
\end{proof}

\begin{lem} \label{lem:spanning-tree-exist} $X$ has a spanning tree.
\end{lem}

\begin{proof} If $H_d(X;\R) = 0$ then $X$ is a spanning tree. If $H_d(X;\R) \ne 0$,
 then we can pick an essential $d$-cell and remove it, decreasing
$\beta_d(X)$ by
one. Repeat this process until $\beta_d$ is zero.
Evidently, the resulting subcomplex $T$ contains
$X_{d-1}$ and by Lemma~\ref{lem:add_or_remove}, we have
$\beta_{d-1}(T) = \beta_{d-1}(X)$. Hence, $T$ is a spanning tree.
\end{proof}

The following is straightforward, and its proof is left to the reader.

\begin{lem} Any spanning tree for $X$ may be
obtained by removing essential $d$-cells.
Furthermore, if $T$ is a spanning tree of $X$,
the  number of essential $d$-cells withdrawn to construct $T$ is
equal to $\beta_d(X)$.
\end{lem}

\begin{lem}
Let $T$ be a spanning tree of $X$ and let $\widetilde{T} = T \cup b$, where $b$
is an essential cell in $\widetilde T$. If $b'$ is an essential $d$-cell of $\widetilde{T}$
different from $b$, then $U := \widetilde{T} \setminus b'$ is a spanning tree.
\end{lem}

\begin{proof}
Since $b'$ is essential, Lemma~\ref{lem:add_or_remove} implies $H_d(U)$ has rank zero.
This lemma also implies $\beta_{d-1}(U) = \beta_{d-1}(\widetilde{T}) = 
\beta_{d-1}(T)$. Since our construction
leaves the $d-1$ skeleton fixed, $U$ is a spanning tree.
\end{proof}

\begin{lem}\label{lem:torsion} Let $T$ be a spanning tree of $X$ and let $b 
\in X_{d} \setminus T_{d}$. Then
$[\partial b]$ generates
a torsion element of $H_{d-1}(T;\Bbb Z)$.
\end{lem}

\begin{proof} Since $T$ is a spanning tree, we have that $b$ is attached 
to $T$ along its attaching map $\partial b \to T$. 
Hence, the homology class $[\partial b]$ lies in $H_{d-1}(T;\Bbb Z)$.
The isomorphism $H_{d-1}(T;\R) \cong H_{d-1}(X;\R)$, along with the fact that
$\partial b$ bounds the cellular chain $b$ in $X$, implies $\partial b$ is torsion in $H_{d-1}(T;\Z)$.
\end{proof}

Recall the linear transformation
$
\bar T\: C_d(X;\R) \to Z_d(X;\Bbb R)
$
defined in the introduction, which was defined on essential cells as
\[
\bar T(b) = c/t_b \, ,
\] where $c$ is a generator of $H_d(T\cup b;\Bbb Z)$
and
\[
t_b = \langle c,b\rangle \, ,
\] where the inner product is taken in $C_d(X;\R)$
(here we are using the inclusion $H_d(T\cup b;\R)\subset C_d(X;\R)$ to make sense
of the inner product). Observe that $|t_b|$ is the order of $[\partial b] \in H_{d-1}(T;\Bbb Z)$.

\begin{lem} \label{prop:ij-proportion}  Let $T$ be a spanning tree of $X$, let
$b_i \in X_d\setminus T_d$ and let $b_j$ be
an essential $d$-cell. Then
\[
   \langle\bar T(b_i),b_j\rangle = \frac{t_{b_j}}{t_{b_i}}\, .
\]
\end{lem}

\begin{proof} $\langle\bar T(b_i),b_j\rangle = \langle c/t_{b_i},b_j\rangle =
1/t_{b_i} \langle c,b_j\rangle = t_{b_j}/t_{b_i}$.
\end{proof}

\begin{cor}If $U$ is a spanning tree obtained by adding and then removing an essential cell from a spanning tree $T$
as above, then
\[
  \langle \bar T(b_i), b_j\rangle  \langle b_i,\bar U(b_j)\rangle = 1 \, .
\]
\end{cor}

\begin{lem}  For an essential $d$-cell $b$, the class $[\partial b] \in H_{d-1}(T;\Z)$ is a torsion element
of order $|t_b|$. In particular, there is a short exact sequence
\[
0\to \Z/t_b \Z \to H_{d-1}(T;\Z) \to H_{d-1}(T\cup b ;\Z) \to 0 \,.
\]
\end{lem}

\begin{proof}  By Lemma \ref{lem:torsion}, $[\partial b]$ is a torsion class.
Let $t$ be its order.

By slight abuse of notation, we let $\partial b$ denote  the cycle representing
 $[\partial b]$.
Then $t\partial b$ is also a  cycle, which  is also the boundary of
a unique integral $d$-chain $w \in C_d(T\cup b;\Z)$. It is straightforward
to check that $tb - w$ is a generator of $H_d(T\cup b;\Z) = Z_d(T\cup b;\Z)$.
Then $\langle tb-w,b\rangle = t$. It follows that $t = \pm t_b$. The short exact sequence
is a direct consequence.
\end{proof}

For a finite CW complex $Y$ of dimension $d$, 
let $\theta_Y$ denote the order of the torsion subgroup of
$H_{d-1}(Y;\Bbb Z)$.

\begin{cor} \label{cor:key} For $T, U, b_i,$ and $b_j$ as above,
\[
  \theta_T^2\langle\bar T(b_i),b_j\rangle = \theta_U^2\langle b_i,\bar U(b_j)\rangle.
\]
\end{cor}

\begin{proof} Set $t_i := t_{b_i}$ and let $Y = T \cup b_i = U \cup b_j$.
Then the exact sequence
\[0 \to \Z/t_i\Z \to H_{d-1}(T;\Z) \to H_{d-1}(Y;\Z)\to 0\] gives
$|t_i|\theta_Y = \theta_T$ and similarly $|t_j|\theta_Y = \theta_U$. Consequently,
\[
\theta_T^2\langle\bar T(b_i),b_j\rangle = \theta^2_{Y} t_it_j =
 \theta_U^2\langle b_i,\bar U(b_j)\rangle \, . \qedhere
\]

\end{proof}

\section{Proof of Theorem \ref{bigthm:network} and Addendum \ref{add:Kirchhoff_formula}
\label{sec:proof-thm-network}}

The proof will proceed  along the lines given in
\cite{NS} in the classical setting.
Given a spanning tree $T$,
let $\{ b_1,\dots,b_k \}$ elements of $X_d \setminus T_d$.

\begin{lem} \label{lem:basis} The collection $\bar T(b_1),\dots, \bar T(b_k)$ is a basis for
$Z_d(X;\R)$.
\end{lem}

 \begin{proof} Recall that $Z_d(X;\R) = H_d(X;\R)$.
 Let $q\: X\to X/T$ be the quotient map. Then the homomorphism $q_\ast\:
 H_d(X;\R) \to H_d(X/T;\R)$ is an isomorphism, and $H_d(X/T;\R)$ is the vector space
 with basis $b_1,\dots,b_k$.
 It's straightforward to check that $q_\ast\circ \bar T\: C_d(X;\R) \to H_d(X/T;\R)$
 maps a $d$-cell
 $b$ to itself when $b \in X_d\setminus T_d$ and is zero otherwise.
 \end{proof}

 \begin{cor}  For any $z\in Z_d(X;\R)$, we have $\bar T(z) = z$.
 \end{cor}

 \begin{proof} Use the  Lemma \ref{lem:basis} to write $z = \sum_i s_i\bar T(b_i)$. Then
 \[\bar T(z)
 = \sum_i s_i \bar T(\bar T(b_i)) = \sum_i s_i\bar T(b_i) = z\, . \qedhere
 \]
 \end{proof}

 \begin{lem} \label{lem:TU} For distinct $d$-cells $b_i,b_j \in X$, let $\cal T_{ij}$ be the set of spanning
 trees such that $\langle \bar T(b_i),b_j \rangle \ne 0$.  Then
 \[
 \sum_{T\in \cal T_{ij}}w_T \langle \bar T(b_i), b_j\rangle_R
 =  \sum_{U\in \cal T_{ji}}w_U \langle \bar b_i, U(b_j)\rangle_R \, .
 \]
 \end{lem}

\begin{proof}  From the definition of of the weights, have
$r_jw_T = r_i w_U$. Note that $\langle \bar T(b_i),b_j\rangle_R = r_j$. Using
 Corollary \ref{cor:key}, we infer
\[
\theta_T^2\langle\bar T(b_i),b_j\rangle_R = \theta_U^2\langle b_i,\bar U(b_j)\rangle_R \, .
\]
Now sum up over all $T\in \cal T_{ij}$.
\end{proof}

\begin{proof}[Proof of Theorem \ref{bigthm:network}]
Consider the operator $F:= \sum_T w_T\bar T$, where the sum is over
all spanning trees of $X$.
For any pair of $d$-cells $b_i$ and $b_j$ of $X$ we have
\begin{align*}
\langle \sum_T w_T\bar T(b_i),b_j\rangle_R
& = \sum_{T\in \cal T_{ij}}   w_T\langle\bar T(b_i),b_j\rangle_R \\
& =\sum_{U\in \cal T_{ji}}w_U \langle \bar b_i, \bar U(b_j)\rangle_R \qquad
\text{ by Lemma \ref{lem:TU} } , \\
& = \langle b_i ,\sum_U w_U \bar U(b_j)\rangle_R \\
& = \langle b_i ,\sum_T w_T \bar T(b_j)\rangle_R
\end{align*}
Hence $F$ is self-adjoint in the modified inner product.

If $z$ is a cycle, then $F(z) = (\sum_T w_T)z =: \Delta z$. Consequently,
$(1/\Delta) F$ restricts to the identity on $Z_d(X;\R)$. As $(1/\Delta) F$
is self-adjoint, it is the orthogonal projection in the modified inner product.
\end{proof}

\begin{proof}[Proof of Addendum \ref{add:Kirchhoff_formula}]  Let $z$ be the orthogonal projection of $R^{-1}\mathbf V$ in the modified inner product. Then $R^{-1}\mathbf V - z \in B^{d}_R(X;\R)$, i.e.,
\[
0 = \langle R^{-1}\mathbf V - z, z'\rangle_R = \langle \mathbf V - Rz, z'\rangle
\]
for all $z'\in Z_{d}(X;\R)$. Hence, $\mathbf V - Rz \in B^d(X;\R)$.
The uniqueness of $z$ is a consequence of the fact that $B^{d}(X;\R)$ is the orthogonal
complement to $Z_d(X;\R)$ in the standard inner product.

The proof of the last part is given by direct calculation using the self-adjointness
of the operator $\sum_T w_T\bar T$:
\begin{align*}
\langle z,b\rangle &= \tfrac{1}{r_b} \langle z,b\rangle_R \, , \\
& = \tfrac{1}{r_b} \langle \tfrac{1}{\Delta}{\textstyle \sum}_T w_T R^{-1}\mathbf V,b\rangle_R \, ,\\
&= \tfrac{1}{\Delta}\sum_T \tfrac{w_T}{r_b} \langle R^{-1}\mathbf V,\bar T(b)\rangle_R \, , \\
&= \tfrac{1}{\Delta}\sum_T \tfrac{w_T}{r_b} \langle \mathbf V,\bar T(b)\rangle\, . \qedhere
\end{align*}
\end{proof}

\section{A weak form of Theorem  \ref{thm:higher-matrix-tree} \label{sec:weak-matrix-tree}}
The goal of this section is to show how Theorem \ref{bigthm:network}
implies Theorem \ref{thm:higher-matrix-tree} up to the identification of the
prefactor $\gamma$. The prefactor will be computed in \S\ref{app:extras}, where in addition
we prove an enhanced version of 
Theorem \ref{thm:higher-matrix-tree}.

Recall the given function $r\: X_d \to \R_+$ of \S\ref{sec:intro}.
It is convenient to set
\[
W := \ln r\: X_d\to \R \, .
\]
 Then $r_b = e^{W_b}$,
and we may also write $R = e^W\: C_d(X;\R) \to C_d(X;\R)$.
Conversely, given any function $W\: X_d\to \R$, we set
$r := e^W\: X_d \to \R_+$. It is convenient to think of
$W$ as lying in $C_d(X;\R)$ by representing it as $\sum_{b\in X_d} W_bb$.

Recall that to each spanning tree $T$ we associated the weight
\[
   w_T = \theta_T^2  \prod_{b\in T_d}r_b^{-1},
\]
where $\theta_T$ is the order of the torsion subgroup of $H_{d-1}(T;\Z)$.

\begin{rem} Let $M$ be a smooth manifold and let
 $V$ be a finite dimensional real vector space. Suppose
$f\: M \to  V$ is a smooth map. Then the directional derivative
defines a $V$-valued, smooth, differential 1-form  $df \in \Omega^1(M;V)$.
In the special case when $M = U$ is a finite dimensional real vector space, then $\Omega^1(M;V)$ can be identified
with the space of smooth maps $U \to \hom(U,V)$.
\end{rem}

Consider the linear operator
\[
\partial\partial^\ast_R = \partial e^{-W} \partial^*: C_{d-1}(X;\R) \to C_{d-1}(X;\R)\, .
\]
Since the image of $\partial\partial^\ast_R$ is contained in $B_{d-1}(X;\R)$,
restriction of this operator to $B_{d-1}(X;\R)$  gives an isomorphism
\begin{equation} \label{calL(W)}
\mathcal{L}(W) \: B_{d-1}(X;\R) @> \cong >> B_{d-1}(X;\R)\, .
\end{equation}
For $R = e^W$, $\cal L(W)$ is the operator $\cal L^R$ defined in \S\ref{sec:intro}.

We can regard $W \mapsto \cal L(W)$ as defining a smooth map
\begin{equation} \label{calL}
\cal L\: C_d(X;\R) \to \End(B_{d-1}(X;\R))\, ,
\end{equation}
which is a family of linear operators parametrized by $C_d(X;\R)$.
To avoid notational clutter, when $W$ is understood, we will often write
$\mathcal L(W)$ without referring to its argument. Therefore,
$\cal L$ can refer to either \eqref{calL} or \eqref{calL(W)}.

\begin{prop} \label{prop:derived} Theorem \ref{bigthm:network} implies
the  identity
\[
  d \ln \det \mathcal L =  d \ln \sum_T w_T  \, .
\]
\end{prop}

\begin{rem} In keeping with our notational ambiguity, the left side of the display in
Proposition \ref{prop:derived}
is to be interpreted as the value at $W$ of $d \ln \det \mathcal L \in \Omega^1(C_d(X;\R);\R)$.

Proposition \ref{prop:derived} is equivalent to the statement
\[
\det \mathcal L  = \gamma\sum_T w_T\, .
\]
for a suitable positive constant $\gamma$, as yet to be determined.
This gives Theorem \ref{thm:higher-matrix-tree}
modulo the determination of the prefactor $\gamma$.
\end{rem}

\begin{proof}[Proof of Proposition \ref{prop:derived}] We take the differential
of the natural logarithm of $\det\cal L$:
\begin{equation} \label{eq:trace-reduction}
\begin{alignedat}{1}
d \ln \det \mathcal{L} &= d \tr \ln \mathcal{L} \\
 &= \tr d(\ln \mathcal{L}) \\
 &= \tr (\mathcal{L}^{-1} d \mathcal{L}),
\end{alignedat}
\end{equation}
where $d\mathcal{L} = \partial d e^{-W} \partial^* = -\partial dW e^{-W} \partial^*$.

The cyclic property of the trace implies
\begin{equation} \label{eq:further-trace-reduction}
\begin{alignedat}{1}
   \tr(\mathcal{L}^{-1} d \mathcal{L}) &= -\tr (\partial dW e^{-W} \partial^* \mathcal{L}^{-1}).
\end{alignedat}
\end{equation}
If we set $A:= e^{-W} \partial^* \mathcal{L}^{-1}\: B_{d-1}(X;\R) \to B^d_R(X;\R)$, then
$\tr(\mathcal{L}^{-1} d \mathcal{L})  = -\tr (\partial dW A) = -\tr(dW A \partial)$.
Consequently,
\begin{equation} \label{eq:trace-sum-expression}
\begin{alignedat}{1}
  d \tr \ln \mathcal{L} &= -\tr (dW A \partial)  \\
  &=  -\sum_{b\in X_{d}}  \langle b| dW A \partial|b \rangle \\
  &= -\sum_{b\in X_{d}}  \langle b | dW A|\partial b \rangle  \\
  &= -\sum_{b\in X_{d}}  dW_b \langle b| A|\partial b \rangle  \, ,
\end{alignedat}
\end{equation}
where $dW_b$ denotes the $b$-coordinate function of $dW$, i.e.,
$dW_b(x) = dW(x)(b) = W(b)$, and
$\langle i|H|j\rangle$ stands for the inner product $\langle i,H(j)\rangle$.

By definition, $A$ is a left inverse to $\partial \: B^d_R(X;\R) \to B_{d-1}(X;\R)$, so the expression $\langle b| A | \partial b \rangle$ is
the same as $\langle b, Pb\rangle$, where $P\: C_d(X;\R) \to B_R^d(X;\R)$ is the
orthogonal projection in the modified inner product $\langle\phantom{a},\phantom{a}\rangle_R$.  By Theorem \ref{bigthm:network}, we have
\begin{equation} \label{eq:B-projection}
P = I - \tfrac{1}{\Delta}\sum_{T}w_T \bar T\, ,
\end{equation}
where $I$ is the identity operator. By inserting this expression into
$\langle b, Pb\rangle$
and doing some rewriting, we obtain
\begin{equation} \label{eq:triple-bracket}
  \langle b | A | \partial b \rangle = \frac{1}{\Delta} \sum_{T,b\in T_d} w_T,
\end{equation}
where $\Delta = \sum_T w_T$ and the displayed sum is over all spanning trees
$T$ for which $b$ lies in $T$. This allows us to rewrite the expression appearing in the last
line of  Eq.~\eqref{eq:trace-sum-expression}
as
\begin{equation} \label{eq:halfway}
\sum_{b \in X_d} dW_b \langle b | A | \partial b \rangle =
\frac1{\Delta} \sum_T \sum_{b \in T_d}w_T  dW_b\, .
\end{equation}
On the other hand, for any spanning tree $T$ we have
\begin{equation} \label{eq:sum-spanning-tree-differential}
d \ln \sum_T w_T = \frac{1}{\Delta} \sum_T dw_T,
\end{equation}
where $d w_T$ is given by
\begin{equation}\label{eq:dwt}
 d w_T = \theta_T^2 \, d \! \prod_{b \in T_d} e^{-W_b} =- \sum_{b \in T_d} dW_b w_T .
\end{equation}
Inserting Eq.~\eqref{eq:dwt} into Eq.~\eqref{eq:sum-spanning-tree-differential} gives
\begin{equation} \label{eq:derived-tree-sum}
d \ln \sum_T w_T = \frac{-1}{\Delta} \sum_T \sum_{b \in T_d} w_T dW_b.
\end{equation}
Assembling equations \eqref{eq:trace-reduction}, \eqref{eq:trace-sum-expression}, \eqref{eq:halfway}, \ref{eq:sum-spanning-tree-differential}
and  \eqref{eq:derived-tree-sum}, we conclude
\[
  d \ln \det \mathcal L = -\frac{1}{\Delta} \sum_T \sum_{b \in T_d} w_T dW_b =  d \ln \sum_T w_T  \, . \qedhere
\]
\end{proof}

\section{The low temperature limit \label{sec:low-temp}}
Here we compute $\det \mathcal{L}$  in the low temperature $\beta \to \infty$ limit
for a certain kind of $W$.
We set
\[
\mathcal{L} = \partial e^{-\beta W} \partial^*\: B_{d-1}(X;\R) \to B_{d-1}(X;\R)
\]
where $W\: X_d \to \R$,
and $\beta \in \R_+$ represents inverse temperature.

 Our freedom in choosing $W$ shows
this determinant will tend
to $\det \mathcal{L}^T$, where $\mathcal L^T$ is $\mathcal L$ restricted to a
spanning tree.
\begin{defn}
Fix a spanning tree $T$ of $X$.
A function $W\: X_d \to \R$ is {\it good} if

\[ W_{\gamma} > \sum_{\alpha \in T_d} W_{\alpha} -
k \min_{\alpha \in T_d} W_{\alpha} \quad \mbox{ for any } \gamma \in X_d \setminus T_d,
\]
where $k$ is the number of $d$-cells of $X$.
\end{defn}

\begin{prop} \label{prop:goodW} For good $W\: X_d \to \R$, we have
\[
\lim_{\beta \to \infty} \frac{\det \mathcal{L}^T}{\det\mathcal L} = 1 \, .
\]
\end{prop}

Before commencing with the proof,
recall the boundary of a $d$-cell $\alpha \in X_d$ is given by
\[
  \partial \alpha = \sum_{\substack{j \in X_{d-1} \\ \langle \partial \alpha, j \rangle \neq 0}}
  b_{\alpha j} j
\]
where $b_{\alpha j} := \langle \partial \alpha, j \rangle$
is the incidence number of $\alpha$ and $j$.
With respect to the standard inner product,
the adjoint operator $\partial^*$ on a $(d-1)$-cell $j$ is given by
\[
  \partial^*j = \sum_{\substack{\alpha \in X_d \\ \langle \partial \alpha, j \rangle \neq 0}}
   b^*_{j \alpha} \alpha
\]
where $b^*_{j \alpha} := b_{\alpha j}$.
A straightforward computation of the matrix elements of $\mathcal{L}$
yields
\begin{equation*}
   \mathcal{L}_{ij} = \sum_{\alpha \in X_d}
   e^{-\beta W_{\alpha}} b_{\alpha i} b_{\alpha j}.
\end{equation*}

\begin{proof}[Proof of Proposition \ref{prop:goodW}]
Since $C_d(X;\R)$ is a real vector space with basis spanned by the set
of $d$-cells, $X_d$, 
we have an orthogonal projection $Q\: C_d(X;\R) \to C_d(T;\R)$,.
This allows us to write $\partial^* = \partial^*_T +
\tilde \partial^*$, where $\partial^*_T$ is defined via the commutative
diagram
\[
  \xymatrix{
  B_{d-1}(X; \R)
 \ar@{->}[r]^{\partial^*}
 \ar@{->}[dr]_{\partial_T^*}
 &
 C_d(X;\R)
 \ar@{->}[d]^{Q}
 \\
&
 C_d(T;\R)\, .
}
\]
It also enables us to write
\[
\cal L = \mathcal L^T + \delta \mathcal L\, ,
\]
where $\cal L^T = \partial\partial^\ast_T$.
 Together, these imply
\begin{equation*}
   \mathcal{L}^T_{ij} = \sum_{\alpha \in T_d}
   e^{-\beta W_{\alpha}} b_{\alpha i} b_{\alpha j}\, .
\end{equation*}
Our choice of good $W$ implies
any $e^{-\beta W_{\gamma}}$ appearing in the expansion of
$\delta \mathcal{L}$ must be less than any $e^{-\beta W_{\alpha}}$ appearing in $\mathcal{L}^T$
and conversely. This also means the matrix elements of $\delta \mathcal L$ can be written as
a similiar sum; the only difference is we instead sum over
$\alpha \in X_d \setminus T_d$.

To simplify taking the limit, we compute the quotient of $\det \mathcal L $ by $\det \mathcal L^T$
and  let $\beta \to \infty$. Since $\det \mathcal L^T \neq 0$, we may write
\[
  \frac{\det( \mathcal L^T +  \delta \mathcal L)}{\det \mathcal L^T} =
  \frac{\det( I + (\mathcal{L}^T)^{-1} \delta \mathcal L) \det \mathcal L^T}{\det \mathcal L^T}.
\]
It suffices to prove that $(\mathcal{L}^T)^{-1} \delta \mathcal L$ tends to the
zero operator as $\beta \to \infty$. Equivalently, it is enough to show that 
 the matrix
elements of  $(\mathcal{L}^T)^{-1} \delta \mathcal L$ converge to zero. The
first bound is of $\mathcal{L}^T_{ij}$:
\begin{align*}
 |\mathcal{L}^T_{ij}| &\leq \sum_{\alpha \in T_{d}} e^{-\beta W_{\alpha}} |b_{\alpha i} b_{\alpha j}| \\
  &\leq e^{-\beta \min_{\alpha} W_{\alpha}} \sum_{\alpha} |b_{\alpha i} b _{\alpha j}|  .
\end{align*}
This can be further bounded by defining $B^T = \max_{ij} \sum_{\alpha} |b_{\alpha i} b_{\alpha j}|$. Hence,
we have
\begin{equation} \label{bound_LijT}
 |\mathcal L_{ij}^T| \leq e^{-\beta(\min_{\alpha \in T_d} W_{\alpha})} B^T .
\end{equation}
The standard formula for the inverse of a matrix gives
\begin{equation}
 \left( (\mathcal L^T)^{-1} \right)_{ij} = \frac{\det \bar {\mathcal L}^T_{ij} }{\det \mathcal L^T} \, ,
\end{equation}
where $\bar A_{ij}$ is the $(i,j)$-th cofactor of $A$.
Using the exact expression for the determinant of $\cal L^T$ appearing in
Eq.~\eqref{eq:detLAT} below, Eq.~\eqref{eq:detLA} below\footnote{There is no circularity here;  Eqs.~\eqref{eq:detLA} and \eqref{eq:detLAT} do not depend
on the material in this section.} and the bound
Eq.~\eqref{bound_LijT} in the case of the cofactor $\bar {\mathcal L}^T_{ij}$,
we obtain the estimate
\[
 \left( (\mathcal L^T)^{-1} \right)_{ij} \leq
\frac{ e^{-\beta ( \min_{\alpha \in T_d} W_{\alpha})(n-1)} (n-1)!B^T}{e^{-\beta \sum_{\alpha \in T_d}
 W_{\alpha}} g_T} \, ,
\]
where $g_T = \det(\partial_T^* \partial_T)$ depends only on $T$.

We bound the elements $\delta \mathcal L$ similarly by
\[
  | \delta \mathcal L_{jk}| \leq e^{-\beta (\min_{\gamma \in X_d \setminus T_d} W_{\gamma})} B^{\tilde T}\, ,
\]
where $B^{\tilde T}$ is defined in the obvious fashion.

Finally, the matrix elements of $(\mathcal L^T)^{-1} \partial \mathcal L$ then satisfy
the following inequality:
\begin{equation*}
\label{ineq:det}
 \left( (\mathcal L^T)^{-1} \partial \mathcal L \right)_{ik} \leq
  \frac{(n-1)! e^{-\beta (\min_{\alpha } W_{\alpha}) (n-1)} (B^T)^{n-1}
 n e^{-\beta \min_{\gamma} W_{\gamma}} B^{\tilde T}}
  {e^{-\beta \sum_{\alpha \in T_d} W_{\alpha}} g_T}\, .
\end{equation*}
Collecting
terms independent of $\beta$ into $N$, we see
\[
 \left( (\mathcal L^T)^{-1} \partial \mathcal L \right)_{ik} \leq
e^{-\beta \left( (n-1)\min_{\alpha} W_{\alpha} - \sum_{\alpha} W_{\alpha} +
\min_{\gamma} W_{\gamma} \right)} N
\]
where $\alpha \in T_d$ and $\gamma \in X_d \setminus T_d$. Our choice of $W$ forces
the matrix elements to 0 as $\beta \to \infty$. Therefore,
\[
 \lim_{\beta \to \infty} \, \frac{ \det \mathcal L}{\det \mathcal L^T} = \det I = 1,
\]
completing the proof.
\end{proof}

\section{A generalized form of Theorem \ref{thm:higher-matrix-tree} \label{app:extras}}

In this section we will identify the prefactor $\gamma$
appearing in Theorem \ref{thm:general-higher-matrix-tree}. We will also
generalize Theorem \ref{thm:higher-matrix-tree} in a significant way.

\subsection*{Covolume} If $A$ is a finitely generated abelian group we let
\[
A_{\Bbb R} :=A\otimes_{\mathbb{Z}} \mathbb{\Bbb R}
\]
 denote its {\it realification,} and
we let $\beta(A)=\dim_{\Bbb R} A_{\Bbb R}$ denote the rank of $A$.
Let $t(A)$ be the order of the torsion subgroup of $A$.

For a homomorphism $\alpha: A\to B$ of abelian groups, we denote
$\alpha_{\Bbb R}\: A_{\Bbb R}\to B_{\Bbb R}$ be the induced homomorphism of real vector spaces.

\begin{defn}
\label{def:quasi} A homomorphism $\alpha: A\to B$ of finitely generated abelian groups is called a {\it real isomorphism} if
the induced homomorphism $\alpha_{\Bbb R}: A_{\Bbb R} \to B_{\Bbb R}$ of real vector spaces is an isomorphism.
\end{defn}

Clearly, $\alpha$ is a real isomorphism if and only if its kernel and its cokernel are finite. If $\alpha$ is a real isomorphism, then $\beta(A)=\beta(B)$, where we recall that
$\beta(A)$ is the rank of $A$.
We will henceforth assume that $A$ and $B$ are free abelian. In this case
$\alpha$ is a real isomorphism if and only if $\alpha$ is a monomorphism with finite cokernel. 

\begin{defn} For $\alpha\: A \to B$ a real isomorphism with $A$ and $B$ free abelian, we
let \[t(\alpha)\in \Bbb N\] denote the order of the cokernel, i.e., $t(\alpha) := t(B/\alpha(A))$.
\end{defn}

An ordered basis
for $A$ determines an ordered basis for $A_{\Bbb R}$, and given any pair of ordered bases for $A$, the associated change of basis matrix for  $A_{\Bbb R}$ has determinant $\pm 1$.
This defines an equivalence relation on ordered bases for $A$ with exactly two distinct equivalence classes. A choice of equivalence class is referred to as an {\it orientation} of $A$.
Consequently, when orientations for $A$ and $B$ are chosen, and  $\alpha\: A \to B$ is a real isomorphism, then the determinant $\det \alpha \in \Bbb R$ is defined and depends
only on the choice of orientations. Furthermore, its absolute value
$|\det \alpha|$ is well defined and does not depend on the choice of orientations.
The latter has the following interpretation: choose an ordered basis for $B$. This defines an inner product on $B_{\Bbb R}$ making the ordered basis for $B$ into an orthonormal basis for $B_{\Bbb R}$. Then $\alpha(A) \subset B_{\Bbb R}$ is a lattice
and $|\det \alpha|$ is its {\it covolume}, that is, the volume of the torus $B_{\Bbb R}/\alpha(A)$ with respect to the induced Riemannian metric, or equivalently, the volume of a fundamental
domain of the universal covering  $B_{\Bbb R} \to B_{\Bbb R}/\alpha(A)$.

\begin{prop}
\label{prop:det-tor} For a real isomorphism $\alpha\: A\to B$ of finitely generated free abelian groups we have $|\det\alpha|=t(\alpha)$.
\end{prop}

\begin{proof} Choose an ordered basis for $B$, and give $B_{\Bbb R}$ the induced inner product.

Consider the inclusions $\alpha(A) \subset B\subset B_{\Bbb R}$. Then we have a finite covering space
\[
B/\alpha(A) \to B_{\Bbb R}/\alpha(A) \to B_{\Bbb R}/B\, ,
\]
in which the covering projection $B_{\Bbb R}/\alpha(A) \to B_{\Bbb R}/B$ is a local isometry and $B_{\Bbb R}/\alpha(A)$ is the fiber over
the basepoint.
This shows that the covolume of $\alpha(A)$ is the product of the covolume of $B \subset B_{\Bbb R}$ with $|B/\alpha(A)| = t(\alpha)$. But the covolume of $B \subset B_{\Bbb R}$ is $1$.
\end{proof}

\subsection*{Generalization of Theorem  \ref{thm:general-higher-matrix-tree}}
Recall that for $W\: X_d\to \R$, we have the operator
\[
\cal L(W) = \partial e^{-W}\partial^\ast\: B_{d-1}(X;\R)@> \cong >> B_{d-1}(X;\R)\,
\]
which is just $\cal L^{R} = \partial\partial^\ast_{R}$,
as defined in the introduction, with  $R = e^W$. Again, we
suppress the argument $W$ from the notation and refer to $\cal L(W)$ as
$\cal L$.

As we showed earlier in Proposition \ref{prop:derived}, we have the following representation:
\begin{eqnarray}
\label{Kirchhoff-matrix} \det{\cal L}=\gamma\sum_{T} w_{T},
\end{eqnarray}
where the constant $\gamma$ is still to be determined.

\begin{defn}
Let $A\subset C_{d-1}(X;\Bbb Z)$ be a subgroup.
Define
a natural number
\[
\mu(A) \in \Bbb N
\]
as follows: let $\{e_i\}$ be a basis for $A$. Consider the matrix $g$ whose $(i,j)$-entry is
given by $g_{ij} = \langle e_i,e_j\rangle$, where the inner product
is taken in $C_{d-1}(X;\R)$.  
Set $\mu(A) := \det g$.
\end{defn}

Since  $e_i$ expressed in the standard basis for $C_{d-1}(X;\R)$
has integer components, we infer that $g_{ij} \in \Z$, so $\mu(A)$ is an integer.
Alternatively, one can define $\mu(A)$ as the {\it square} of the covolume of the lattice $A \subset A_{\Bbb R}$ given by restricting
the standard inner product of $C_{d-1}(X;\R)$ to $A_{\R}$. The equivalence of the two definitions can be seen as follows:
let $B$ be the matrix whose rows are
the vectors $e_i$ expressed in an orthonormal basis for $C_{d-1}$. 
Then $|\det B|$ is the covolume
of $A \subset A_{\R}$. Furthermore, $g = BB^\ast$, so $\mu(A) = \det g = (\det B)^2 \in \Bbb N$.

For any abelian group $U$, we set
\[
B^U_{d-1} := B_{d-1}(X;U)\, ,
\]
that is, the image of the boundary operator $\partial\: C_d(X;U) \to C_{d-1}(X;U)$ of the cellular chain complex of $X$ with $U$ coefficients.

\begin{hypo_no}\label{hypo:projection}  The inclusion $A \subset C_{d-1}(X;\R)$
is such that the orthogonal projection
$P_{A}: B_{d-1}^{\R} \to A_{\Bbb R}$ is induced by a real isomorphism
 $p_{A}: B^{\Bbb Z}_{d-1}\to A$, i.e., $P_{A}=({p}_{A})_{\R}$.
 \end{hypo_no}
 
 Consider the composite operator
 \[
 {\cal L}_{A}: A_{\Bbb R} @> \cong >> A_{\Bbb R}
 \]
 defined by ${\cal L}_{A}=P_{A}\partial e^{-W}\partial^{*}|_{A_{\Bbb R}}$.

\begin{thm}[Generalized Higher Weighted Matrix-Tree Theorem] \label{thm:general-higher-matrix-tree}
We have
\begin{eqnarray}
\label{Kirchhoff-matrix-2}
\det{\cal L}_{A}=\gamma_{A}\sum_{T}w_{T}\, ,
\end{eqnarray}
where the prefactor is given by
\begin{eqnarray}
\label{Kirchhoff-matrix-normalization} \gamma_{A}=\frac{\mu(A) t(p_{A})^{2}}{\theta^2_X}.
\end{eqnarray}
\end{thm}

\begin{rem} The choice $A= B_{d-1}(X;\Bbb Z)$ gives Theorem \ref{thm:higher-matrix-tree}.
\end{rem}

\begin{rem} If $A = A_{S}$ is the free abelian group generated by a
judiciously chosen subset $S \subset X_{d-1}$, we will obtain $\mu(A_{S})=1$.
Using this choice of $A$ as well as $W = 0$,
Theorem \ref{thm:general-higher-matrix-tree}
gives a generalization of the main result of \cite{Petersson} to CW complexes.
\end{rem}

\begin{proof}[Proof of Theorem \ref{thm:general-higher-matrix-tree}]
As above, we have
\[
\cal L := \partial\partial_R^\ast = \partial e^{-W}\partial^\ast\: B_{d-1}(X;\R) @> \cong >> B_{d-1}(X;\R)\, .
\]
Then
\[
{\cal L}_{A} = P_{A}{\cal L}P_{A}^{*}\, ,
\]
 which implies
\begin{equation} \label{eq:detLA}
 \det{\cal L}_{A}=\det({\cal L})\det(P_{A}P_{A}^{*})\, .
\end{equation}
If we apply this to Eq.~(\ref{Kirchhoff-matrix}), we reproduce Eq.~(\ref{Kirchhoff-matrix-2}) with $\gamma_{A}=\gamma\det(P_{A}P_{A}^{*})$. It suffices to identify the prefactor $\gamma_A$.

Consider  the operator ${\cal L}^{T}$ for some spanning tree. We have
\begin{equation} \label{eq:detLAT}
\det{\cal L}^{T}=\det({\partial}_{T}e^{-W}{\partial}_{T}^{*})=\det({\partial}_{T}{\partial}_{T}^\ast e^{-W})=\tfrac{w_{T}}{\theta^2_T}\det({\partial}_{T}{\partial}_{T}^\ast)
\end{equation}

Applying Eq.~\eqref{Kirchhoff-matrix-2} in the case of good $W$ and in the low temperature limit, the left hand side of that equation
 tends to the determinant of the operator ${\cal L}_{A}^{T}$ for the spanning tree $T\subset X$ of  maximal weight, whereas the right hand side is dominated by the single contribution associated with the same spanning tree $T$. Consequently, Eq.~\eqref{eq:detLAT} implies
\begin{eqnarray}
\label{Kirchhoff-matrix-3}
\det\left(\partial_{T}\partial_{T}^\ast\right)\det\left(P^T_{A}(P^T_{A})^{*}\right) =
\gamma_{A}\theta^{2}_T.
\end{eqnarray}
Since $P^T_{A}= (p^T_{A})_{\R}$, where the real isomorphism
$p_{A}^{T}: B_{d-1}(T;\Bbb Z))\to A$ is obtained by composing the real isomorphism $p_{A}: B_{d-1}(X;\Bbb Z)\to A$ with the inclusion $B_{d-1}(T;\Bbb Z)\subset B_{d-1}(X;\Bbb Z)$, we have
\[
\det\left(P_{A}P_{A}^{*}\right)=\mu(A)(\mu(B_{d-1}(T;\Bbb Z)))^{-1}(\det p_{A}^{T})^{2}\, .
\]
 We further note that, since $T$ is a spanning tree, the free abelian
 group $B_{d-1}(T;\Z)$ has basis  $\{\partial e_{1},\dots,\partial e_{s}\}$, where $e_{1},\ldots,e_{s}$ are the $k$-cells of $T$, so that we have a matrix $g$ of inner products with the matrix elements $g_{ij}=\langle\partial_{T} e_{i},\partial_{T} e_{j}\rangle =\langle \partial_{T}^{*}\partial_{T}e_{i},e_{j}\rangle $, which implies $\mu(B_{d-1}(T;\Bbb Z))=\det(\partial_{T}^{*}\partial_{T})$. Then Eq.~(\ref{Kirchhoff-matrix-3}) assumes the form
 \[
 \mu(A)(\det p_{A}^{T})^{2}=\gamma_{A}\theta^2_T\, .
 \]
 Combining this with Proposition~\ref{prop:det-tor} results in
\begin{eqnarray}
\label{gamma-expression}
\gamma_{A}= \frac{\mu(A)t(p_{A}^{T})^{2}}{\theta^2_T}.
\end{eqnarray}
The right side of Eq.~\eqref{gamma-expression} is written in
terms of a particular spanning tree $T$, however, it does not actually depend on
this choice. An invariant expression that does not contain $T$ is obtained by using the following relations:
\begin{eqnarray} \label{eq:t-theta-relation}
\label{relations} \frac{t(p_{A}^{T})}{t(p_{A})}=t(B_{d-1}(X;\Bbb Z)/B_{d-1}(T;\Bbb Z))=\frac{\theta_T}{\theta_X}.
\end{eqnarray}
Substituting Eq.~(\ref{relations}) into Eq.~(\ref{gamma-expression}) results in an invariant expression for $\gamma_{A}$, given by Eq.~(\ref{Kirchhoff-matrix-normalization}).
\end{proof}

\subsection*{Alternative forms of Theorem \ref{thm:higher-matrix-tree}}
In this subsection we deduce Addendum \ref{addC} as well as a generalization of it to the weighted case. 
Let us now return to the more general situation of Theorem \ref{thm:general-higher-matrix-tree}.

\begin{thm} \label{thm:LA-decomp} With $A \subset C_{d-1}(X;\Z)$ as above, we have
\[
\det \cal L_A = \sum_T \det \cal L_A^T \, .
\]
\end{thm} 

\begin{proof} Using Eq.~\eqref{eq:t-theta-relation}  we infer that
\[
\gamma_A = \frac{\mu(A)t(P_A)^2}{\theta^2_X} =  \frac{\mu(A)t(P^T_A)^2}{\theta^2_T}
\] 
for any spanning tree $T$.
Combining this with Theorem \ref{thm:general-higher-matrix-tree} in the case of 
a spanning tree $T$ we obtain
\[
\det \cal L^T_A = \gamma_A w_T \, .
\]
The conclusion now follows by summing over all $T$.
\end{proof}

In the special case when $A = B_{d-1}(X;\Z)$, Theorem \ref{thm:LA-decomp}
reduces to

\begin{cor}  \label{cor:Lmu} $\det \cal L = \sum_T \det \cal L^T = \sum_T \mu_T$.
\end{cor}

\section{Reidemeister torsion and Theorem \ref{thm:torsion-tree} \label{sec:torsion-tree}}

\subsection*{Reidemeister torsion} 
Milnor \cite{Milnor} defined the Reidemeister torsion of a not necessarily acyclic
finite chain complex over a field equipped with the auxiliary structure of an ordered
basis of its chains as well as a choice of ordered basis of its
homology groups. In this section we restrict ourselves to the
case of torsion for chain complexes defined
over the real numbers. 

Consider the case of a chain complex $C_\ast$ of finite dimensional vector spaces
over $\R$ having non-trivial terms in degrees $0 \le \ast \le d$. Let
 $\partial \: C_k \to C_{k-1}$ be the boundary operator.
 Let $Z_k \subset C_k$ be the subspace of $k$-cycles and let $B_k \subset Z_k$ the subspace
of $k$-boundaries. We also set $H_k = Z_k/B_k$. 

We then have short exact sequences
\[
0 \to Z_k \to C_k \to B_{k-1} \to 0 \qquad \text{ and } \qquad 
0 \to B_k \to Z_k \to H_k\to 0\, .
\]
If we choose splittings  $s_{k-1}\:B_{k-1} \to C_k$ and $t_k\:H_k \to Z_k$,
we are entitled to write
$C_k \cong Z_k \oplus B^k \cong B_k \oplus H_k \oplus B^k$.

Pick bases
${\mathfrak b_k} := \{b_k^i\}, {\mathfrak c}_k := \{c_k^i\}, 
{\mathfrak h_k} := \{h_k^i\}$ for $B_k, C_k$, and $H_k$, respectively.
It follows that $\{ s_{k-1}(b_{k-1}^i), t_k(h_k^i), b_k^i\}$
forms another basis for $C_k$. Let $\{\mathfrak b_{k}\mathfrak h_k \mathfrak b_{k-1}\}$ denote this basis
and let 
\[
[\mathfrak b_{k}\mathfrak h_k \mathfrak b_{k-1}/\mathfrak c_k]
\]
denote the change of basis matrix that expresses the basis 
$\mathfrak b_{k}\mathfrak h_k \mathfrak b_{k-1}$ in terms of
the basis $\mathfrak c_k$. 
Let $\mathfrak c = \{\mathfrak c_k\}$ and $\mathfrak h = \{\mathfrak h_k\}$.

\begin{defn}[Milnor {\cite[p.~365]{Milnor}}] The {\it torsion} of the pair $(C_\ast,\mathfrak h)$ is defined by 
\[
  \tau(C_\ast) = 
  \prod_{k \geq 0}  \det[\mathfrak b_{k}\mathfrak h_k \mathfrak b_{k-1}/\mathfrak c_k]  ^{(-1)^k}\, ,
\]
which is consistent with Milnor's definition with respect to the identification
of $K_1(\R) \cong \R^\times$ given by the determinant function.
\end{defn}

Milnor shows that the definition is independent of the choice of $\mathfrak b$ as well as
the splittings. Thus the torsion is really an invariant of the triple $(C_\ast,\mathfrak c,\mathfrak h)$.

In what follows,  $C_\ast = C_\ast(X;\R)$ is the cellular chain complex of a finite connected CW complex $X$ which has a preferred basis consisting of the set of cells. In this case, we think of the torsion as an invariant of the pair $(X,\mathfrak h)$ and we set
\[
\tau(X;\mathfrak h) := \tau(C_\ast(X;\R)) \, ,
\]
where we have indicated in the notation the dependence on the choice of homology basis.
It will be useful to single out a specific kind of homology basis. Let $H_\ast(X;\Z)_0 \subset
H_\ast(X,\R)$ be the lattice given by taking
 the image of the evident homomorphism $H_\ast(X;\Z) \to H_\ast(X;\R)$. Note that
 $H_\ast(X;\Z)_0$ has a preferred isomorphism to the torsion free part of $H_\ast(X;\Z)$.

\begin{defn} A {\it combinatorial basis} for $H_\ast(X;\R)$ consists of a basis for
$H_k(X;\Z)_0$ for $k \ge 0$. 
\end{defn}

Henceforth we fix a combinatorial basis $\mathfrak h$.
Let $r \: \amalg_k X_k \to \R_+$ be a positive-valued function on the
set of cells of $X$.  As in previous sections, we write $R\: C_\ast(X;\R) \to C_\ast(X;\R)$ for the linear 
transformation determined by $b\mapsto r_b b$, and $R = e^W$. We have 
a modified inner product $\langle b,b'\rangle_R = \langle r_b b,b'\rangle$.
We also have an operator
\[
\cal L_k(W) = \partial \partial^*_R  := \partial e^{-W_{k+1}} \partial^* e^{W_k}\: B_k(X;\R) \to B_k(X;\R)\, ,
\]
where $\partial_R^*$ is the formal adjoint to $\partial\: C_{k+1}(X;\R) \to C_k(X;\R)$
in the modified inner product on both source and target.
We define $H^R_k(X;\R)$
to be the orthogonal compliment of $B_k(X;\R)$ in $Z_k(X;\R)$ with respect
to modified inner product on $C_k(X;\R)$, and we then have a preferred identification
$H^R_k(X;\R) \cong H_k(X;\R)$ given by sending a cycle to its homology class. 
As in the introduction, we let $\eta_k$ be
the square of the covolume of $H_k(X;\Z)_0 \subset H^R_k(X;\R)$, with
respect to the basis $\mathfrak h_k$ for $H_k(X;\Z)_0$ and the inner
product on $H^R_k(X;\R)$ obtained by restricting the modified inner product on $C_k(X;\R)$.\footnote{In the introduction, $\eta_k$ was defined only in the case when $W=0$; the current
notation applies to an arbitrary $W$.}

\begin{thm}\label{thm:tor-det-L} 
Let $X$ be a finite, connected CW complex. Then
\begin{equation*}
\tau^2(X;\mathfrak h) = \frac{\prod_{k \mbox{ } \mathrm{even}} \det \cal L_k(W)}{\prod_{k
\mbox{ }\mathrm{odd}}
\det \cal L_k(W)} \cdot \frac{\prod_{k \mbox{ } \mathrm{odd}, b \in X_k} e^{W_{kb}}}{\prod_{k
\mbox{ } \mathrm{even},
b \in X_k} e^{W_{kb}}} \cdot \frac{\prod_{k \mbox{ } \mathrm{even}} \eta_k}{\prod_{k
\mbox{ } \mathrm{odd}}
\eta_k}  
 \, . 
\end{equation*}
\end{thm}

\begin{rem}  If we take $W = 0$, then  Theorem \ref{thm:tor-det-L} immediately implies that 
$\tau^2(X;\mathfrak h)$ is an invariant of the lattice $H_\ast(X;\Z)_0 \subset H_\ast(X;\R)$
rather than just an invariant of the specific choice of combinatorial basis $\mathfrak h$.
Since this lattice doesn't depend on any choices, we infer that 
$\tau^2(X;\mathfrak h)$ depends only on the CW structure of $X$. In fact, the method of proof of \cite[th.~7.2]{Milnor} shows that $\tau^2(X;\mathfrak h)$ is invariant under subdivision. 
\end{rem}

\begin{proof}[Proof of Theorem \ref{thm:tor-det-L}]  For the purpose of this proof we suppress
$W$ and write $\cal L = \cal L(W)$. We also set $C_\ast := C_\ast(X;\R)$.
Define the splitting maps $s_{k-1}: B_{k-1} \to C_k$ by
\[ s_{k-1}(b_{k-1}^i) = e^{-W_k} \partial^* e^{W_{k-1}} \cal L^{-1}_{k-1} (b^i)
= \partial^*_R \cal L^{-1}_{k-1} ( b_{k-1}^i)\, .\]
Let $B_R^k(X;\R)$ denote the image of $s_{k-1}$, and similarly we define
$B^k_R(X;\Z)$ to be $s_{k-1}(B_k(X;\Z))$.
Note that $B^k_R(X;\R)$ is the
orthogonal compliment to $Z_k$ in the modified inner product on $C_k$.

Let $\gamma^k$ denote the square of the covolume of $B^k_R(X;\Z)\subset B^k_R(X;\R)$,
using the inner product on $B^k_R(X;\R)$ induced by the modified
inner product on $C_k$. Similarly, let
$\gamma_{k-1}$ denote the square of the covolume of $B_{k-1}(X;\Z) \subset B_{k-1}(X;\R)$,
where $B_{k-1}(X;\R)$ is given the inner product by restricting the modified 
inner product on $C_{k-1}$. Using the isomorphism $B_k \oplus H_{k} \oplus B_{k-1} \overset{\cong}\to C_k$
determined by the splitting, we infer
\begin{equation} \label{eq:basis-change}
 \det[\mathfrak b_{k}\mathfrak h_k \mathfrak b_{k-1}/\mathfrak c_k]^2 =  
 \frac{\gamma_k \eta_k \gamma^k }{\prod_{b \in X_k} e^{W_{kb}}}\, ,
\end{equation}
so the square of the Reidemeister torsion is
\begin{equation} \label{eqn:squared-torsion}
\tau^2(X;\mathfrak h) = \frac{ \prod_{k \text{ even}} \gamma_k \eta _k  \gamma^k}{\prod_{k
\text{ odd}}
\gamma_k  \eta_k  \gamma^k} \frac{ \prod_{k \text{ odd}, b \in X_k} e^{W_{kb}}}{
\prod_{k \text{ even}, b \in X_k} e^{W_{kb}}} \, .
\end{equation}
Since $s_k = \partial^*_W \cal L_{k}^{-1}$, we have $s_k^* = \cal L_k^{-1} \partial$
(since $\mathcal{L}$ is self-adjoint). Therefore, $s_k^*s_k = \cal L_k^{-1} \partial
 \partial^*_W \cal L_k ^{-1} = \cal L_k^{-1}$. We use this fact to compute the
quotient of $\gamma^k/\gamma_{k-1}$. Recall that $\gamma^k$ is given 
as the determinant of the inner product matrix, we compute 
\begin{align*}
  \langle s_{k-1}(b_{k-1}^i), s_{k-1}(b_{k-1}^j) \rangle_R
  &= \langle s^*_{k-1}s_{k-1}(b_{k-1}^i), (b_{k-1}^j) \rangle_R \\
  &= \langle \cal L_{k-1}^{-1} (b_{k-1}^i), (b_{k-1}^j) \rangle_R .
\end{align*}
The determinant of the matrix with these entries latter is, by definition, $(\det U)^2 \det \cal L_{k-1}$, where
$U$ is the change of basis matrix
expressing $\mathfrak b_{k-1}$ in terms of an orthornormal basis 
for $B_{k-1}(X;\R)$ in the modified inner product.
 A similar observation shows that the determinant of the matrix whose entries are
$\langle b_{k-1}^i, b_{k-1}^j \rangle_R$ is $(\det U)^2$, and this is just $\gamma_{k-1}$.

Consequently, the quotient of these determinants is
\begin{equation} \label{eqn:gamma-quotient}
\frac{ \gamma^k}{\gamma_{k-1}} = \frac{1}{\det \cal L_{k-1}} \, .
\end{equation}
Inserting Eqn.~\eqref{eqn:gamma-quotient} into Eqn.~\eqref{eqn:squared-torsion} and
performing the evident cancellations, we conclude
\begin{equation*}
\tau^2(X;\mathfrak h) = \frac{\prod_{k \mbox{ } \mathrm{even}} \det \cal L_k(W)}{\prod_{k
\mbox{ }\mathrm{odd}}
\det \cal L_k(W)} \cdot \frac{\prod_{k \mbox{ } \mathrm{odd}, b \in X_k} e^{W_{kb}}}{\prod_{k
\mbox{ } \mathrm{even},
b \in X_k} e^{W_{kb}}} \cdot \frac{\prod_{k \mbox{ } \mathrm{even}} \eta_k}{\prod_{k
\mbox{ } \mathrm{odd}}
\eta_k}. \qedhere
\end{equation*}
\end{proof}

In the special case  when $W=0$,
we can combine Theorem \ref{thm:tor-det-L} with Corollary \ref{cor:higher-matrix-tree}.
This immediately  gives Theorem \ref{thm:torsion-tree}:

\begin{cor}[Torsion-Tree Theorem] \label{thm:tor-det-L-w=0}
For a finite, connected CW complex $X$, we have
\begin{equation*}
\tau^2(X;\mathfrak h) = 
{\prod_{k \ge 0} 
 (\delta_{k}\sum_{T\in \cal T_{k+1}} \theta^2_T})^{(-1)^k}\,  .
\end{equation*}
\end{cor}

\subsection*{An alternative formula} 
In this part we shall derive a different formula for the torsion in terms of a single
spanning tree in each degree as well as a choice of auxiliary structure, namely,
homology truncation data for $X$.


\begin{hypo_no}
For each $k \ge 1$, we fix a spanning tree $T^k$ for $X^{(k)}$.
Our convention is to set $T^0 =\emptyset$.
\end{hypo_no}

\begin{defn} A {\em homology truncation} of $X$ in degree $k \ge 0$, subordinate to
$T^k$, is a subcomplex $i\: V^k \subset X^{(k)}$ 
such that $T^k \subset V^k$ and $i_\ast\: H_\ast(V^k;\R) \to H_\ast(X;\R)$ are isomorphisms for $\ast \le k$. 
\end{defn}

In induction argument similar to the proof of Lemma \ref{lem:spanning-tree-exist} shows that homology truncations exist. Note that $V^0$ consists
of a single vertex of $X$. We have a filtration
\[
T^0 \subset V^0 \subset \cdots \subset X^{(k-1)} \subset T^k \subset V^k \subset X^{(k)} \subset \cdots
\]


\begin{lem}\label{lem:Tk-split}
The choice of spanning tree $T^k$ determines a splitting $B_{k-1}(X;\R) \to C_k(X;\R)$.
The choice of homology truncation $V^k$ subordinate to $T^k$ determines a splitting $H_k(X;\R) \to Z_k(X;\R)$.
\end{lem}

\begin{proof}  The first splitting is the composition
\[
B_{k-1}(X;\R) = B_{k-1}(T^k;\R) @> \partial^{-1} >\cong > B_k(T^k;\R) @>>> C_k(X;\R) 
\]
and the second is given by 
\[
H_k(X;\R) @> i_*^{-1} >\cong > H_k(V^k;\R) = Z_k(V^k;\R) @> i_\ast >> Z_k(X;\R)\, . \qedhere
\]
\end{proof}

Define a basis for $B^k(X;\Z)$, $\mathfrak b^k = \{b^k_i\}$, as given by the cells of
$T^k_k$. Here we are using the preferred isomorphism $B^k(X;\Z) \cong C_k(T^k;\Z)$. 
This defines a
basis $\mathfrak b_{k-1}$  for $B_{k-1}(X;\R) $ by $\{ b^i_{k-1} = \partial b^k_i \}$. The
basis for homology in degree $k$ is the combinatorial basis $\mathfrak h_k$ given as an input to the torsion. As always, the basis for $C_k(X;\R)$ is given
by the set of $k$-cells.

Before explicitly identifying the torsion, note that in each dimension
$k$ there are essentially three types of cells:
\[
 X_k = (T_k^k) \cup (V_k^k \setminus T_k^k) \cup (X_k \setminus V_k^k)\, .
\]
Roughly speaking, the first set of cells contributes to $B^k$, the
second set contributes to $H_k$ and the last set contributes to $B_k$.
This gives us a decomposition of the $k$-chains
\begin{equation}\label{eq:splitting}
 C_k(X;\R) = C_k(T^k;\R) \oplus C_k(V^k/T^k;\R) \oplus
C_k(X/V^k;\R)\, .
\end{equation}
(when $k = 0$, we replace $C_0(V^0/T^0;\R)$ with $C_0(V^0,T^0;\R) = \R$, etc.)

We first identify the homological contribution in degree $k$ to the 
torsion. With respect to the splitting Eq.~\eqref{eq:splitting}, the combinatorial basis $\mathfrak h_k$ 
has image contained in the direct sum
\[
C_k(T^k;\R) \oplus C_k(V^k/T^k;\R) = C_k(V^k;\R)\, .
\]
Hence, its contribution to the torsion is left invariant if we project these
elements 
onto $C_k(V^k/T^k;\R) = H_k(V_k/T_k;\R) = H_k(X;\R)$ 
(since the other summand $C_k(T^k;\R) =B^k(T^k;\R)$ maps to
$B^k(X;\R)$ and the relevant determinant remains unchanged if we project away from  
$B^k(X;\R)$). Consequently, the homological contribution to the torsion in degree $k$ is
given by the determinant of the composite
\[
H_k(X;\R) @> i_\ast^{-1} > \cong > H_k(V^k;\R) @> p_* >\cong > H_k(V^k/T^k;\R)\, ,
\]
where $p\: V^k \to V^k/T^k$ is the quotient map.
So we wish to identify $\det p_\ast /\det i_\ast$.

\begin{defn} Let 
\[
\chi_k\in \Bbb N
\] 
denote the square of the determinant
of $i_\ast\: H_k(V^k;\R) \to H_k(X;\R)$, i.e., the square of the covolume of the lattice
$i_*(H_k(V^k;\Z)) \subset H_k(X;\R)$.
\end{defn}

Applying Proposition \ref{prop:det-tor} to the real isomorphism
$H_k(V^k;\Z) \to H_k(V^k/T^k;\Z)$, we infer

\begin{lem}\label{prop:det-TbarK}
The determinant of  $p_\ast$ is the ratio
$\pm \theta_{T^k}/\theta_{V^k}$.

\end{lem}

Consequently, up to sign, the contribution of $\mathfrak h_{k}$ to the determinant defining the Reidemeister torsion is 
\begin{equation} \label{hk}
\frac{\theta_{T^k}}{\theta_{V^k} \sqrt{\chi_k}}\,.
\end{equation}



We next identify the contribution in degree $k$ to the torsion provided by 
the basis $\mathfrak b_k$. As defined above this basis is given by the boundaries
of the cells of $T_{k+1}$. This leads us to consider the composite
\begin{equation}
 C_{k+1}(T^{k+1};\Z) @> \partial>>  B_k(T^{k+1};\Z)
@>q_k >> C_k(X /V^k; \Z) \, ,
\end{equation}
where $q_k$ is induced by the quotient map $T^{k+1} \to X/V^k$.
The homomorphism $\partial$ is an isomorphism and so it has determinant $\pm 1$.
The second homomorphism $q_k$ is a real isomorphism and therefore the 
determinant of its realification, $\det((q_k)_\R)$, has value $\pm t(q_k)$ by 
 Proposition \ref{prop:det-tor}. Note that $(q_k)_\R$ 
 is the restriction of the orthogonal projection $C_k(X;\R) \to C_k(X/V^k;\R)$
 to the subspace $B_k(T_{k+1};\R) \subset C_k(X;\R)$.
and the projection of $\mathfrak b_k$ onto this summand gives its contribution to the torsion.
Hence, the determinant of the composition $(q_k)_{\R}\circ \partial$ is $\pm t(q_k)$.
So the contribution in degree $k$ of  $\mathfrak b_k$  to the torsion is $\pm t(q_k)$.

Lastly, the contribution to the torsion in degree $k$ provided by the basis $\mathfrak b_{k-1}$
is given by the standard basis of $C_k(T_k;\R)$ via the splitting Eq.~\eqref{eq:splitting}. 
It is then evident that the contribution in degree $k$ of $\mathfrak b_{k-1}$  to the torsion 
is $1$.

Assembling, we obtain
\begin{equation}
\det[\mathfrak b_k \mathfrak h_k\mathfrak b_{k-1}] = \pm t(q_k) \cdot \frac{\theta_{T^k}}{\theta_{V^k} \sqrt{\chi_k}} \cdot 1\, .
\end{equation}
Forming the square of the Reidemeister torsion, we conclude

\begin{thm} \label{thm:last} For a connected, finite CW complex $X$ with combinatorial
homology basis $\mathfrak h$, spanning tree data $\{T^k\}$ and homology truncation data
$\{V^k\}$, we have
\[
\tau^2(X;\mathfrak h) = \prod_{k \geq 0} \left( \frac{\theta_{T^k}^2t(q_k)^2}{\theta_{V^k}^2\chi_k } 
 \right)^{(-1)^k}\, ,
\]
where $q_k\: B_k(T^{k+1};\Z)\to  C_k(X /V^k; \Z)$ and $\chi_k\in \Bbb N$ are as above.
\end{thm}

\begin{ex} If $X$ has dimension one, then all terms appearing in Theorem \ref{thm:last} are equal to one. Hence, $\tau^2(X;\mathfrak h) = 1$ whenever $X$ is a connected finite graph.
\end{ex}

\begin{ex} Let $X = \R P^2$. Then we may choose $T^1 = \ast = V^1$ and $T^2 = \R P^2 = V^2$.
In this instance, the only non-trivial term appearing in Theorem \ref{thm:last} is $t(q_1)^2 = 4$. Hence $\tau^2(\R P^2;\mathfrak h) = \tfrac{1}{4}$.
\end{ex}


\begin{thebibliography}{CKS2}
\bibliographystyle{invent}


\bibitem[B]{Bollobas}%
Bollob\'as, B.: Modern graph theory.
\newblock Graduate Texts in Mathematics, 184. Springer-Verlag, New York, 1998.

\bibitem[C]{Chaiken}%
Chaiken, S:
A combinatorial proof of the all minors matrix tree theorem. 
\newblock {\it SIAM J. Algebraic Discrete Methods} {\bf 3} (1982), 319-Ð329. 


\bibitem[CKS1]{CKS1}%
Chernyak, V.Y., Klein, J.R., Sinitsyn, N.A.:
Quantization and Fractional Quantization of Currents in Periodically Driven Stochastic Systems I: Average Currents. 
\newblock {\it J. Chem. Phys.} {\bf 136}, 154107 (2012).

\bibitem[CKS2]{CKS2}%
Chernyak, V.Y., Klein, J.R., Sinitsyn, N.A.:
Quantization and Fractional Quantization of Currents in Periodically Driven Stochastic Systems II: Full Counting Statistics.
\newblock {\it J. Chem. Phys.} {\bf 136}, 154108 (2012).

\bibitem[CKS3]{CKS3}%
Chernyak, V.Y., Klein, J.R., Sinitsyn, N.A.:
Algebraic topology and the quantization of fluctuating currents.
\newblock arXiv preprint 1204.2011.


\bibitem[DKM]{DKM}%
Duval, A.M., Klivans, C.J., Martin, J.L.:
Cellular spanning trees and Laplacians of cubical complexes.
\newblock {\it Adv. in Appl. Math.} {\bf 46} (2011), 247--274.

\bibitem[E]{Eckmann}%
Eckmann, B.: Harmonische Funktionen und Randwertaufgaben in einem Komplex.
\newblock{\it Comment. Math. Helv.} {\bf 17} (1945), 240--255.



\bibitem[Ka]{Kalai}%
Kalai, G.:
Enumeration of $\Bbb Q$-acyclic simplicial complexes.
\newblock {\it Israel J. Math.} {\bf 45} (1983), 337--351.

\bibitem[Ki1]{Kirchhoff1}%
Kirchhoff, G.R.: Ueber den Durchgang eines elektrischen Stromes durch eine Ebene, insbesondere durch eine kreisf\"ormige.
\newblock {\it Annalen der Physik und Chemie} {\bf LXIV} (1845), 497--514.

\bibitem[Ki2]{Kirchhoff2}%
Kirchhoff, G.R.:   \"Uber  die  Auf\"osung  der  Gleichungen,  auf  welche  man  bei  der  Untersuchung  der linearen  Verteilung galvanischer Str\"ome gefuhrt wird.
\newblock {\it Ann. Physik Chemie} {\bf 72} (1847), 497--508;

\newblock On the solution of the equations obtained from the investigation of  the linear distribution of  Galvanic  currents (J.B.~O'Toole, tr.)
{\it IRE Trans. Circuit Theory} {\bf 5} (1958), 4--8.

\bibitem[L]{Lyons}%
Lyons, R.:
Random complexes and $\ell^2$-Betti numbers.
\newblock {\it J. Topol. Anal.} {\bf 1} (2009),  153--175.

\bibitem[M]{Milnor}%
Milnor, J.: Whitehead torsion.
\newblock {\it Bull. Amer. Math. Soc.} {\bf 72} (1966) 358--426.
\bibitem[NS]{NS} %
Nerode, A., Shank H.: An algebraic proof of Kirchhoff's network theorem,
\newblock {\it Amer. Math. Monthly} {\bf 68}, (1961) 244--247.

\bibitem[Mo]{Moon}%
Moon, J.~W.: Counting labelled trees. 
\newblock From lectures delivered to the Twelfth Biennial Seminar of the Canadian Mathematical Congress (Vancouver, 1969). Canadian Mathematical Monographs, No.\ 1, 
Canadian Mathematical Congress, Montreal, Que. 1970

\bibitem[P]{Petersson}%
Petersson, A: Enumeration of spanning trees in simplicial complexes.
\newblock {\it Uppsala University preprint,} May 18, 2009.

\bibitem[RS]{RS}%
 Ray, D.~B., Singer, I.: M. R-torsion and the Laplacian on Riemannian manifolds.
 \newblock{\it Advances in Math.} {\bf 7}, (1971) 145--210.


\bibitem[R]{Roth}%
Roth, J.P.:
An application of algebraic topology to numerical analysis: on the existence of a solution to the network problem.
\newblock{\it Proc. Nat. Acad. Sci. U.S.A.} {\bf 41} (1955), 518--521


\bibitem[W]{Weyl}%
Weyl, H.: Repartici\'on de corriente en una red conductora.
\newblock {\it Revista Matematica Hispano-Americana,} {\bf 5}, (1923) 153--164.

\newblock Distribution of Current in a Conducting Network
(J.~M. Garduno trans., Paul Penfield, Jr.~ed.) 
Imperial College of Science and Technology, 
Department of Electrical Engineering; May 9, 1967.

\end{thebibliography}
\end{document}